\documentclass[11pt]{article}
\usepackage{amsfonts, amsmath}
\usepackage{amssymb}
\usepackage{amsthm,amscd}
\usepackage{enumerate}
\usepackage{tikz}
\usepackage{lipsum}
\usepackage{subfig}
\usepackage{lineno}
\usepackage{caption}
\usepackage{float}

\usepackage[utf8]{inputenc}
\usepackage{color}
\usepackage{subfig}
\usepackage{url}
\usepackage{graphicx}
\usepackage{amsmath}
\usepackage{amsthm}
\usepackage{bm}
\usepackage[plainpages=false]{hyperref}
\usepackage{enumerate}
\usepackage{datetime}

\usepackage{titlesec}
\titleformat{\section}
	{\normalfont\Large\bfseries\filcenter}{\thesection.}{1 ex}{}
\titleformat{\subsection}
	{\normalfont\normalsize\bfseries}{\thesubsection.}{1 ex}{}
\titleformat{\subsubsection}[runin]
	{\normalfont\normalsize\bfseries\filcenter}{\thesubsubsection.}{1 ex}{}

\usepackage[charter]{mathdesign}
\usepackage[mathcal]{eucal}

\usepackage[margin=1.2 in]{geometry}

\usepackage[square,numbers,sort&compress]{natbib}

\usepackage{thmtools,thm-restate}
\declaretheorem[within=section]{theorem}
\declaretheorem[sibling=theorem]{lemma}

\declaretheorem[sibling=theorem]{corollary}

\newcommand\RR{\mathbb{R}}

\newcommand{\defn}[1]{\emph{\color{blue} #1}} 

\newcommand\p{{\bf p}}

\newcommand\q{{\bf q}}

\begin{document}
\title{Universal rigidity on the line, point order}
\author{Bryan Chen\thanks{\texttt{bryangingechen@gmail.com}Partially supported by NSF grant PHY-1554887},
 Robert Connelly\thanks{Department of Mathematics, Cornell University. \texttt{rc46@cornell.edu}. Partially supported by NSF grant DMS-1564493.}, 
Anthony Nixon\thanks{School of Mathematics and Statistics, Lancaster University. \texttt{a.nixon@lancaster.ac.uk}.}, and 
Louis Theran\thanks{School of Mathematics and Statistics, University of St Andrews. \texttt{lst6@st-and.ac.uk}}}
\date{}
\maketitle 
\begin{abstract}
We show that universal rigidity of a generic bar-joint framework $(G,\p)$ in $\RR^1$ depends on more than the 
ordering of the vertices.  In particular, we construct examples of $1$-dimensional generic frameworks 
with the same graph and ordering of the vertices, such that one is universally rigid and one is not.
This answers, in the negative, a question of Jordán and Nguyen.  

Underlying our examples are insights about how universal rigidity behaves under 
projections.  Using these ideas, we also give a simple proof that 
universal rigidity is invariant under affine transformations.
\end{abstract}

\section{Introduction}\label{sec:introduction}
Throughout $G = (V,E)$ will be a finite, simple graph with $n$ vertices and $m$ edges;
we usually assume that $V = \{1, 2, \ldots, n\}$ for convenience.
We will consider the rigidity and flexibility properties of realisations of graphs in Euclidean spaces.

\paragraph{Frameworks}
A \defn{framework} is a pair $(G,\p)$ consisting of a 
graph $G$ with $n$ vertices, $m$ edges and a \defn{configuration} $\p = (p(1), \ldots, p(n))$ of $n$ points in $\mathbb{R}^d$.
A $d$-dimensional framework is in \defn{general position} if no $d+1$ points are affinely dependent. More strongly the framework is \defn{generic} if the set of coordinates of  $\p$   form an algebraically independent set over $\mathbb{Q}$.

Two frameworks $(G,\p)$ and $(G,\q)$ are said to be \defn{equivalent} if they have the same edge lengths:
\[
    \|p_i-p_j\|=\|q_i-q_j\| \qquad \text{(for all edges $ij\in E$).}
\]
More strongly, frameworks 
$(G,\p)$ and $(G,\q)$ are \defn{congruent} if all pairwise distances are the same: 
\[
    \|p_i-p_j\|=\|q_i-q_j\| \qquad \text{(for all vertices  $i,j\in V$).}
\]
This is equivalent to requiring that  $\p$   and  $\q$ be related by a Euclidean isometry.

In what follows, we will be interested in frameworks that are 
not affinely spanning. We will say that 
$(G,\p)$ is a \defn{framework in dimension $d$} (or in $\RR^d$) if the affine span of  $\p$   is $d$-dimensional.  
By picking an appropriate basis, of any $\RR^D$ with $D\ge d$, we can identify the affine span of  $\p$   with $\RR^d$.

\paragraph{Flavors of rigidity}
A framework $(G,\p)$ in $\mathbb{R}^d$ is: \defn{rigid} if every equivalent framework in a neighbourhood of  $\p$   in $\mathbb{R}^d$ is congruent to $(G,\p)$; \defn{globally rigid} if every equivalent framework $(G,\q)$ in $\mathbb{R}^d$ is congruent to $(G,\p)$; and \defn{universally rigid} if every equivalent framework $(G,\q)$ in $\mathbb{R}^D$, for any $D\geq d$,  is congruent to $(G,\p)$.

Given a framework, it is a computationally challenging problem to determine if it is rigid, globally rigid or universally rigid. 
The situation improves for generic frameworks. 
In the case of rigidity one may linearise the problem and consider infinitesimal rigidity 
and the rigidity matrix (see \cite{Wlong} for detailed background on these concepts). 
Asimow and Roth \cite{AR} used this to show that rigidity is a \defn{generic property}. 
That is, either every generic framework for a given graph is rigid or 
every generic framework is flexible. It immediately follows that, generically, rigidity is a 
property of the underlying graph.

\paragraph{Equilibrium stresses}
To understand global rigidity, Connelly \cite{C83} developed the stress matrix which we now describe. 
By ordering the $m$ edges of $G$, we can identify $\RR^E$ with $\RR^m$.
An \defn{equilibrium stress} is a vector $\omega \in \mathbb{R}^m$ such that:
\begin{equation}\label{eq: stress}
    \sum_{j\neq i} \omega_{ij} (p_i-p_j)=0
\end{equation}
for all $1\leq i \leq n$, where $\omega_{ij}$ is taken to be equal to zero if $\{i,j\}\notin E$. Equilibrium stresses are 
vectors in the cokernel of the rigidity matrix (see, e.g.,  \cite{Whiteley}).
Given an equilibrium stress $\omega$ for a framework $(G,\p)$ we define the \defn{stress matrix} 
$\Omega(\omega)$ to be the $n\times n$ symmetric
matrix with off-diagonal entries $-\omega_{ij}$ and diagonal entries $\sum_j \omega_{ij}$.  

As a linear map, 
the stress matrix evaluates the l.h.s. of \eqref{eq: stress} for $1$-dimensional configurations. 
The rank of an equilibrium stress is the rank of its stress matrix.
The all ones vector is in the kernel of the stress matrix, since certainly \eqref{eq: stress} is satisfied when 
all the points are the same.  If  $\q$  is a $d$-dimensional configuration,
we say that  $\q$  satisfies $\omega$ if each of its $d$ coordinate projections are in the kernel of $\Omega(\omega)$.
Hence, the rank of any equilibrium stress for a $d$-dimensional framework is at most $n - d - 1$ (but it might be less).


Connelly showed \cite{C05} that every generic framework $(G,\p)$ in $\mathbb{R}^d$ with a stress matrix of rank $n-d-1$ is globally rigid. Subsequently, Gortler, Healy and Thurston \cite{GHT10} showed that the stress matrix being full rank is also a necessary condition for generic global rigidity for all graphs on at least $d+2$ vertices. (A graph on at most $d+1$ generic vertices has no non-zero equilibrium stress and is globally rigid if and only if it is complete.) It follows that global rigidity is also a generic property 
and hence generic global rigidity depends only on the graph.

\paragraph{Universal rigidity and genericity}
The situation for universal rigidity is different. Even in $\RR^1$ it 
is very easy to notice graphs which are universally rigid at some generic 
configurations and not universally rigid at other generic configurations. 
Perhaps the easiest example is the cycle $C_n$, $n\geq 4$. 
Here one may make a universally rigid framework of $C_4$, on vertices $p_1,p_2,p_3,p_4$, 
in $\mathbb{R}^1$ by choosing  $\p$   such that $\|p_1-p_2\|=\|p_2-p_3\|+\|p_3-p_4\|+\|p_4-p_1\|$. 
If instead the positions of $p_1,\dots, p_4$ alternate along the line then it is easy to see that 
the framework will be flexible in $\mathbb{R}^2$. 
A \defn{stretched cycle} is a framework in $\mathbb{R}^1$, where the underlying graph is a 
cycle $p_1,p_2,\dots,p_n$, and  $p_1< p_2 < \dots p_n$.  
The smaller edges are ``stretched" along the longer edge $\{p_1, p_n\}$.  
Figure \ref{fig:stretched} shows a stretched $4$-cycle. 

\begin{figure}[ht]
\centering
\includegraphics[scale=0.5]{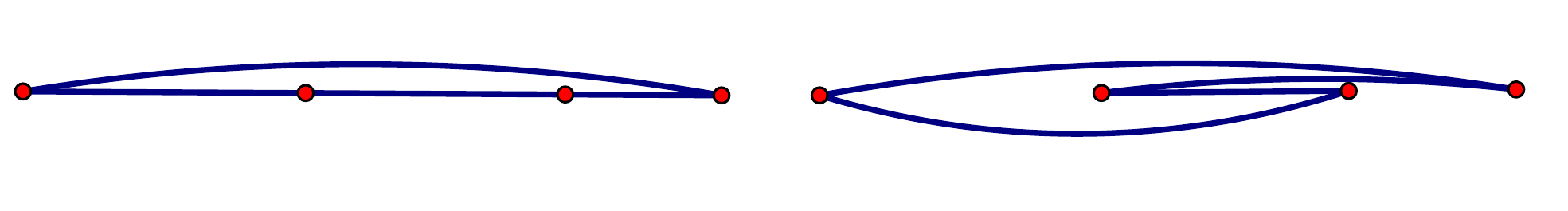}
\captionsetup{labelsep=colon,margin=1.3cm}
\caption{The figure on the left is a \emph{stretched cycle}, while the one the right is  not stretched.  Both are cycles in the line.  Part of the edges are lifted up or down in order to better see the graph.}
\label{fig:stretched}
\end{figure}

\paragraph{Results and organization}
Jord\'{a}n and Nguyen \cite[Question 2.5]{JN15} posed the following question: 
\begin{center}
    {\it 
    Does the  universal  rigidity  of  a  general  position framework $(G, \p)$ in $\mathbb{R}^1$ depend 
    only on the ordering of vertices on the line (and not on the coordinates)? 
    }
\end{center}

The purpose of this short article is to answer their question in the negative. In fact we show the stronger fact that the universal rigidity and the super stability of a \emph{generic} 
framework on the line does not depend solely on the ordering of the vertices. 
We do this by explicitly constructing generic frameworks of the triangular prism graph on the line
that are universally rigid and not universally rigid despite having the same vertex ordering.
Our examples can be extended to an infinite family of graphs for which the ordering does not suffice. 

On the other hand Jord\'{a}n and Nguyen \cite[Theorem 2.2]{JN15} showed that the ordering is sufficient for 
any bipartite graph. Along the way, we show that universal rigidity is 
preserved under orthogonal projections. We can use this result to give a simple 
proof that universal rigidity is affinely invariant.

\section{Universal rigidity}\label{sec:univ}
Let $(G,\p)$ be a framework in $\mathbb{R}^d$.
We say that the edge directions of $(G,\p)$ are \defn{on a conic at infinity} if there is a non-zero 
symmetric $d\times d$ matrix  $\q$ such that, for all edges $i, j \in E$, $p_i-p_j\in C(Q)$, where $C(Q)=\{x\in \mathbb{R}^d: x^tQx=0\}$.
Moreover $(G,\p)$ is \defn{super stable} in $\mathbb{R}^d$ if there exists an equilibrium stress $\omega$ for $(G,\p)$ 
such that $\Omega(\omega)$ is PSD of rank $n-d-1$, where $n$ is the number of vertices of $G$, and the 
edge directions of $(G,\p)$ do not lie on a conic at infinity.
At this point, it's helpful to remind the reader than super-stability implies that 
 $\p$   has $d$-dimensional affine span.

Connelly \cite{C83} proved the following.
\begin{theorem}\label{thm:super}
Let $(G,\p)$ be a super-stable framework in $\mathbb{R}^d$. Then $(G,\p)$ is universally rigid in $\mathbb{R}^d$.
\end{theorem}
Notice that super-stability does not need genericity or even general position.  The only 
non-degeneracy requirement is avoiding a conic at infinity.

To set up some notation, we 
define the \defn{configuration matrix}  $\p$   to be the $d$-by-$n$ matrix whose columns consist of 
the coordinates of a configuration of $n$ points in $\mathbb{R}^d$. We also define the 
\defn{augmented configuration matrix}, $\hat P$, to be the $(d+1)$-by-$n$ matrix obtained from  $\p$   by 
adding a row of $1$'s at the bottom.  Given the configuration  $\p$  , and a stress  $\omega$, 
if the rows of $\hat{P}$ span the co-kernel of the associated stress matrix $\Omega$, we say that the 
configuration  $\p$   is \defn{universal} for the stress $\omega$.

A vector $\omega \in \mathbb{R}^m$ is an equilibrium stress for $(G,\p)$ if and only if $\hat{P}\Omega=0$. 
Hence, 
if a universal configuration for $\omega$ has $d$-dimensional affine span, then the cokernel of 
$\Omega$ 
is $(d+1)$-dimensional, and the rank of $\Omega$ is $n-(d+1)$.  Conversely, if the rank of
$\Omega$ is $n-(d+1)$, then any universal configuration  $\p$   for $\omega$ has a  
$d$-dimensional affine span. 

We recall a standard lemma.  
\begin{lemma}\label{affine} Suppose that the framework $(G,\q)$ with $n$ vertices in $\RR^d$ is in 
equilibrium with respect to the stress $\omega$ with corresponding stress matrix $\Omega$, 
and  $\p$   is another configuration for $G$ whose affine span is $\RR^D$ in 
equilibrium with respect to $\omega$ such that the rank of $\Omega$ is $n-D-1$.  
Then  $\q$ is an affine image of  $\p$  .
\end{lemma}
We will also use the following theorem of Connelly. See Figure \ref{fig:Desargues} for an example.

\begin{theorem}[{\cite[Theorem 5]{C83}}] \label{thm:con83}
Let $(G,\p)$ be a framework in $\mathbb{R}^2$, so that  $\p$   is in strictly convex position
and every edge of the convex hull of  $\p$   is an edge of $G$; we call these edges 
boundary edges and the rest of the edges interior.
If $\omega$ is an equlibrium stress of $(G,\p)$ that is positive 
on the boundary edges and negative on the interior edges, then 
the associated stress matrix is PSD of rank $n - 3$.
\end{theorem}

\section{Projections of universally rigid frameworks}
We make a brief digression into projections of frameworks that 
are universally and globally rigid.

\begin{theorem}\label{thm:global}  If $(G,\hat \p)$ is a framework  that is 
globally rigid in $\RR^D$,  then the orthogonal projection 
of $(G,\hat \p)$ onto $(G,\p)$ in $\RR^d$, for any $d\le D$, is globally rigid in $\RR^d$.
\end{theorem}

\begin{proof}  Let $h_i$ be the orthogonal projection of $\hat p_i$ into the orthogonal complement of 
$\mathbb{R}^d$ in $\mathbb{R}^D$ so that $\hat p_i=p_i+h_i$, and $h_i$ and $p_i$ are orthogonal for all $i$.  Then computing distances we have
\[
\|\hat p_i -\hat p_j\|=\sqrt{\|p_i-p_j\|^2+\|h_i-h_j\|^2}.
\]
If  $\q$ is another configuration corresponding to  $\p$   in $\mathbb{R}^d$, then define $\hat q_i=q_i+h_i$.  Computing distances again we get 
\[
\|\hat q_i -\hat q_j\|=\sqrt{\|q_i-q_j\|^2+\|h_i-h_j\|^2}.
\]
Then we see that 
$\|\hat p_i -\hat p_j\|=|\hat q_i -\hat q_j\|$ if and only if $\|p_i -p_j\|=\|q_i -q_j\|$.  So $(G,\p)$ is equivalent to $(G,\q)$  if and only if $(G,\hat \p)$ is equivalent to $(G,\hat q)$.  So if $(G,\hat \p)$ is globally rigid in $\mathbb{R}^D$, so is $(G,\p)$ in $\mathbb{R}^d$.
\end{proof}

\begin{corollary}\label{cor:global}  If $(G,\hat \p)$ is a framework  that is 
universally rigid in $\RR^D$,  then the orthogonal projection 
of $(G,\hat \p)$ onto $(G,\p)$ in $\RR^d$, $D\ge d$, is universally rigid in $\RR^d$.
\end{corollary}

\begin{proof}  Universal rigidity of $(G,\p)$ in $\RR^d$ is just global rigidity applied to $(G,\p)$ in $\RR^D$ for $D$ large.  So apply Theorem \ref{thm:global} to $\RR^D \times \RR^k \rightarrow \RR^d \times \RR^k$, for $k$ large extending the projection $\RR^D \rightarrow \RR^d $.
\end{proof}
\vspace{0.25cm}

\noindent Several remarks are in order.
\vspace{0.25cm}

\noindent {\bf 1.} If  the projection $(G,\p)$ is universally rigid in $\RR^d$ it is not necessarily true that $(G,\hat \p)$ is globally rigid in $\RR^D$.  
For example a quadrilateral in the plane can project to a stretched cycle in the line, which is super stable and universally 
rigid, while the quadrilateral is not even rigid in the plane.\\

\noindent{\bf 2.}  Figure \ref{difficient-stress.fig} shows the orthogonal projection of a super stable framework in $\RR^3$  
into the plane that only has a one-dimensional stress coming from the stress in $\RR^3$, since each vertex is of degree $3$, 
all the $3$ edges at each vertex have only one linear relation among them as vectors, and the edges are connected. 
So the stress in any one edge determines the stress in all the other edges.   Thus, in the plane, the rank of its 
stress matrix will be $8-3-1=4$, which is  not  maximal.   By Corollary \ref{cor:global} it is universally rigid, but 
not super stable.  See \cite{iterative} for a general discussion of universally rigid but not super stable frameworks.\\

\begin{figure}[ht]
\centering 
\includegraphics[scale=0.3]{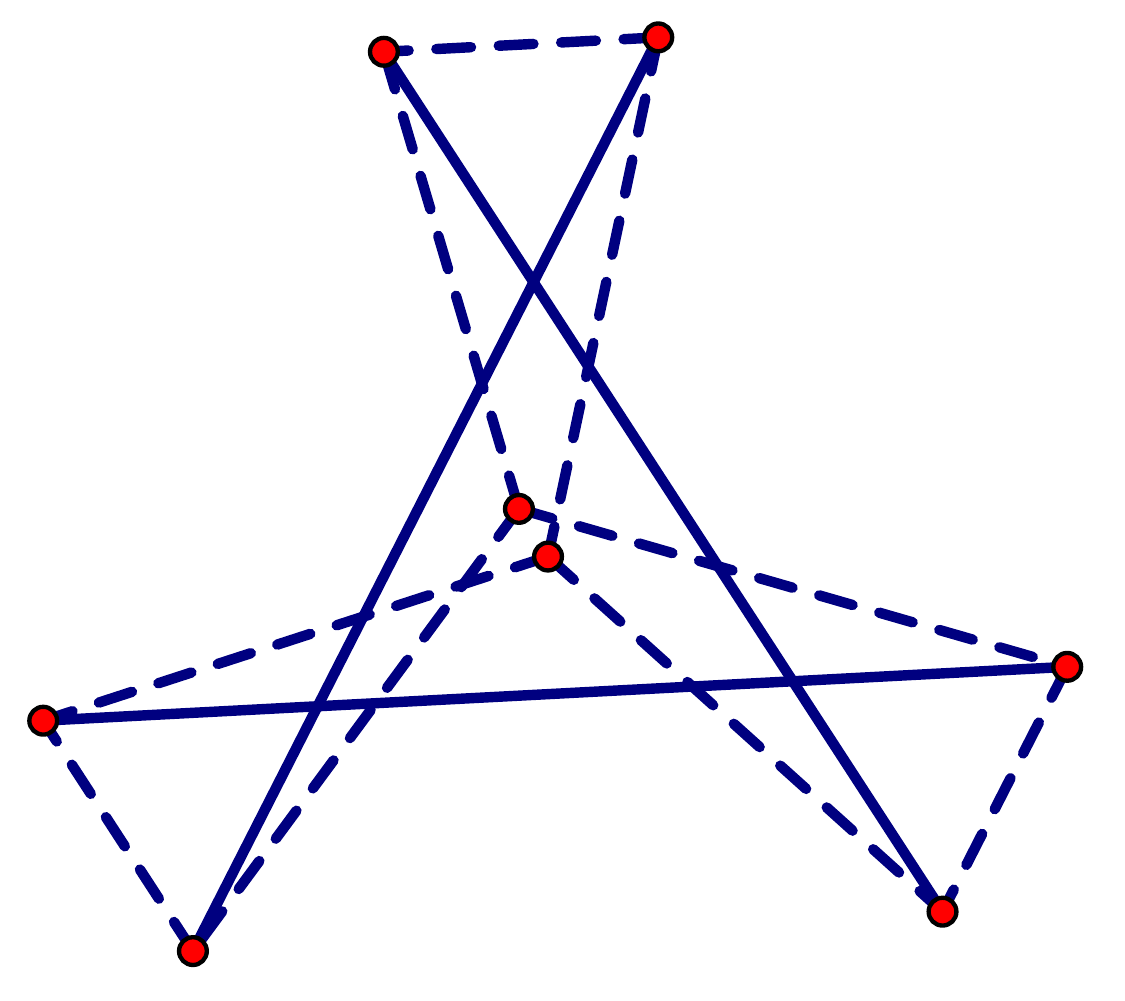}
\captionsetup{labelsep=colon,margin=1.3cm}
\caption{This is a framework in $\RR^2$ that is the orthogonal projection from a super stable framework in $\RR^3$, but is not itself super stable.  See \cite{Combining, C83}.  The dashed edges have a positive stress, and the solid edges have a negative stress. One $Y$-shaped set of dashed edges are $30^{\circ}$ apart from the other.  The two center vertices have been moved slightly to see that they are distinct. }
\label{difficient-stress.fig}
\end{figure}

\noindent {\bf 3.} Corollaries of Theorem \ref{thm:global} can be extended to other flavors of rigidity, such as local or infinitesimal 
rigidity.  

\noindent {\bf 4.}
We also give a short proof that universal rigidity is preserved under affine 
transformations, even singular affine transformations.  Plainly translation preserves universal rigidity.  Now suppose that $(G,\p)$ is a universally rigid 
framework in $\RR^d$ and $A$ is a linear transformation on $\RR^d$.  
Using the singular value decomposition we can write $A$ as $UDV$ where $U$ and $V$ are orthogonal transformations and $D$ is a diagonal 
matrix with non-negative entries on the diagonal.

For the moment, assume that $D$ has all entries in $[0,1]$.  By embedding $(G,\p)$ in $\RR^D$, for sufficiently large $D$,
we can model the action of $D$ by an orthogonal projection onto an appropriately chosen $d$-dimensional
subspace $X$ of $\RR^d$.  If $D$ has entries that are greater than one, we first scale  $\p$   and then 
use the same argument.

Finally, we model the map $A$ by a product of maps $\tilde{U}P\tilde{V}$ where $\tilde{U}$ is orthogonal and 
acts as $U$ on $\RR^d$,  $\p$   is the scaling and projection described above, and $\tilde{V}$ is orthogonal and 
acts as $V$ on $X$.  By Theorem \ref{thm:global} each of these preserves universal rigidity.\\

\noindent {\bf 5.} A related fact is that, by \cite[Corollary 4.8]{CGT17}, super stability is invariant with respect to 
invertible projective 
transformations in $\RR^d$ that do not send any vertices to infinity. 
By contrast, we know that invertible projective 
transformations do \emph{not} always preserve universal rigidity.  
See \cite[Figure 8 and Figure 9]{CGT17}, which is closely related to Figure \ref{fig:1} below.\\

\noindent {\bf 6.} Corollary \ref{cor:global} and the previous remarks give us a test for universal rigidity.  
Suppose that $(G,\p)$ is a $d$-dimensional framework that has a PSD equilibrium stress $\omega$  of rank $n - D -1$, 
where $D\ge d$.  If a $D$-dimensional framework 
$(G,\hat \p)$ satisfies $\omega$ and its edges are not on a conic at infinity, then $(G,\p)$ is universally rigid, 
since it is an affine image of a super stable framework. A necessary and sufficient test is given in 
\cite{iterative}, but this one is simpler to apply.

\section{The triangular prism}\label{sec:prism}

Here we show that universal rigidity on the line, even when the configuration coordinates are generic do not only depend on the ordering of the vertices. This answers Question 2.5 of \cite{JN15} negatively. The fundamental example is the \defn{triangular prism}, the cartesian product of $K_3$ and $K_2$ (depicted in Figure \ref{fig:1}), which we denote by $T$.

We will create a PSD equilibrium stress that shows the universal rigidity of the bottom framework of Figure \ref{fig:1}, and yet the top framework of Figure \ref{fig:1} is flexible in the plane.  (See Figure 9 in \cite{iterative}, where it is shown that any flex of the top framework must have a $2$-dimensional affine span.)

\begin{figure}[ht]
\centering
\includegraphics[scale=0.5]{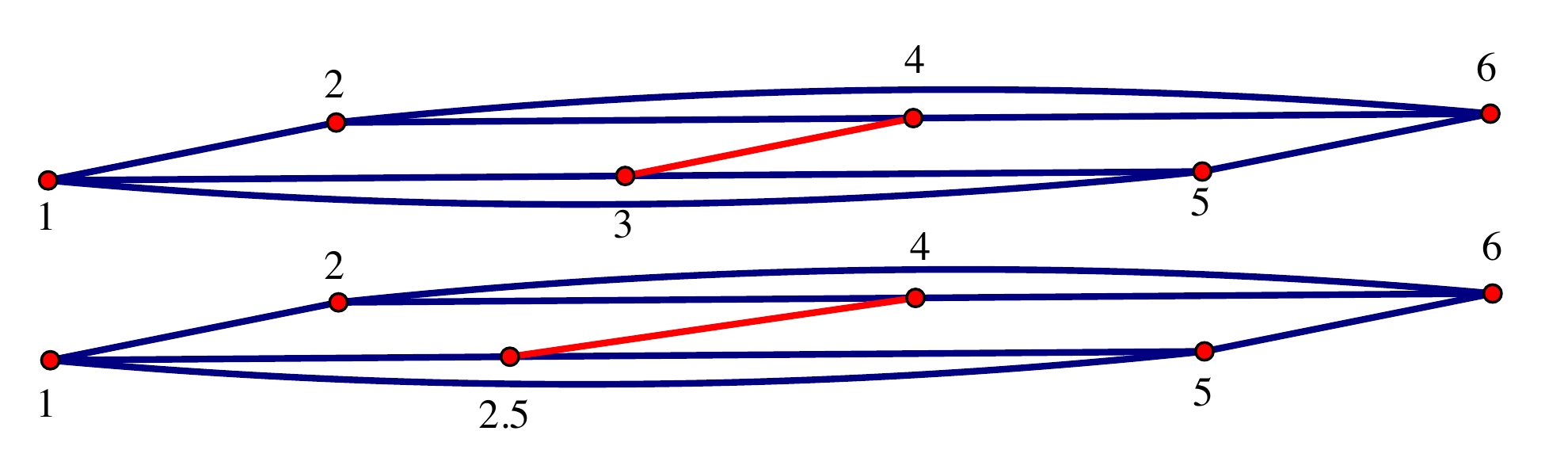}
\captionsetup{labelsep=colon,margin=1.3cm}
\caption{The whole figure is in a line but it is displaced slightly up to see the connections, and the curved arc indicates certain pairs of points that are connected.  The top figure is flexible in the plane as a $3$-rung ladder, whilst the bottom figure is (universally) rigid, that is, it cannot open up.  The ordering of the vertices of both examples in the line are the same.}
\label{fig:1}
\end{figure}


\paragraph{Universally rigid example: first proof}
We begin with the framework $(G,\hat \p)$ in $\mathbb{R}^2$ shown in Figure \ref{fig:Desargues}, which is super-stable by 
Theorem \ref{thm:con83} and hence has a PSD equilibrium stress $\hat{\omega}$ of rank $n - 3$. We project it orthogonally onto the line to a framework $(G,\p)$.  On the line any one of the flat triangles in $(G,\p)$ has a PSD equilibrium stress $\omega$. Plainly $\omega + \hat{\omega}$ is a PSD equilibrium stress and it is satisfied by $(G,\p)$.

Let  $\q$ be any configuration that is universal for 
$\omega + \hat{\omega}$.  By Lemma \ref{affine},  $\q$ is an affine image of $\hat \p$. The affine map sending $\hat \p$ to  $\q$ is determined by its action on the triangle, and we conclude that  $\q$ has $d$-dimensional affine span. Since the edges of $(G,\p)$ are not on a conic, we conclude that the framework $(G,\p)$ is super-stable.  We do not yet know that it is generic.  However, since the projection is infinitesimally rigid, a small perturbation will also be super-stable and generic (see \cite{CGT17}), with the vertices in the same order.

\paragraph{Universally rigid example: second proof}
Our second analysis of the universally rigid example uses affine invariance of universal rigidity.
Starting from the super-stable framework  in Figure \ref{fig:Desargues}, we have a $6$-dimensional 
space of nearby convex frameworks that are also super-stable.  We freely specify the two left 
vertices of the quadrilateral and the bottom right one; this fixes the position of the 
top right one.  Then we freely place the bottom vertex, which leaves $1$ degree of freedom for the top 
one.  Removing isometies gives us $6$ dimensions.

Hence, an affine projection onto a generic line, such as a small wiggle of the $x$-axis is 
generic as a configuration of points in $\RR^1$.  As discussed in the previous section, this 
projection is universally rigid.  Since it is generic, results of \cite{Gortler-Thurston-universal}
tell us it is, in fact, super-stable.


\begin{figure}[ht]
\centering 
\includegraphics[scale=0.3]{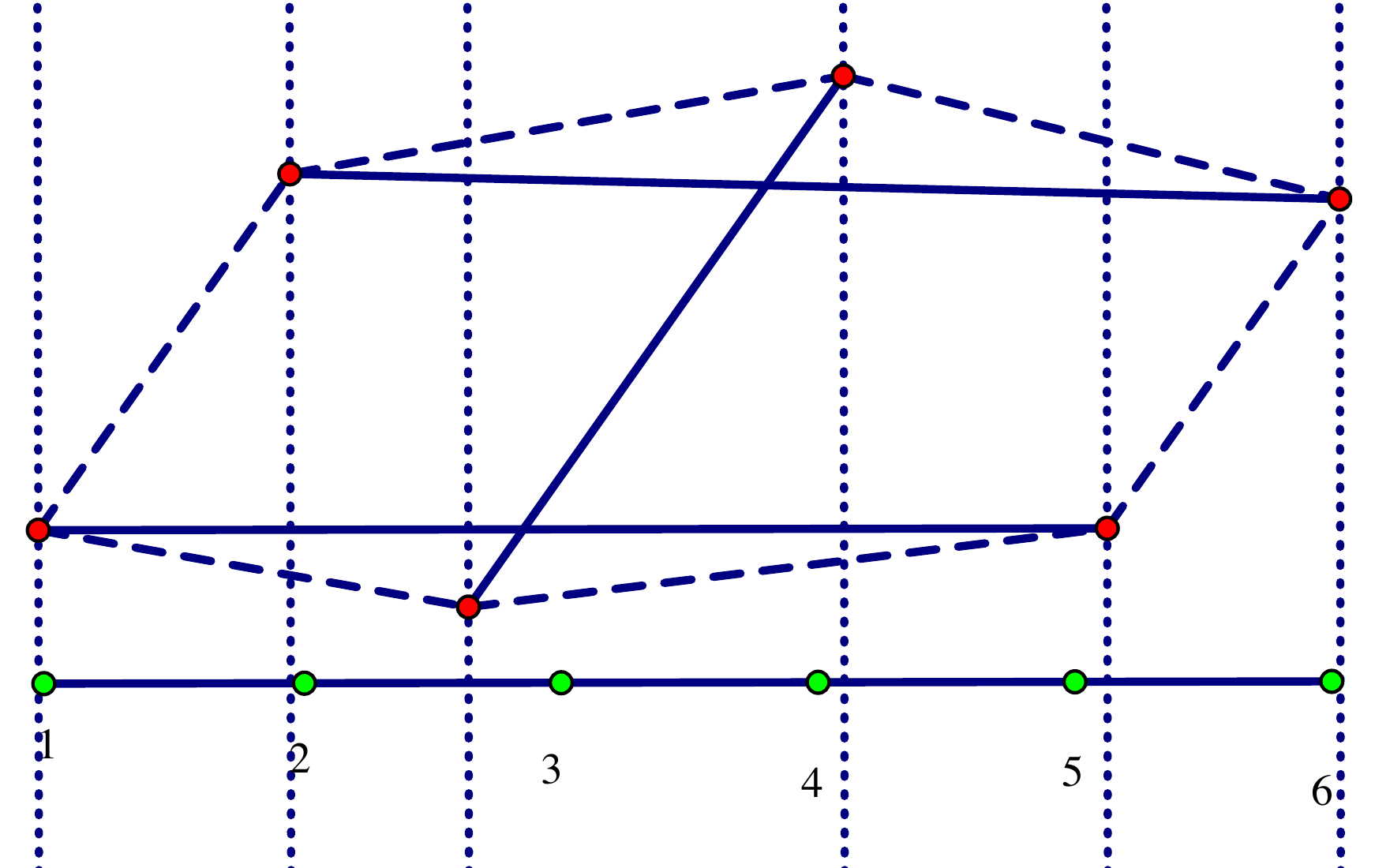}
\captionsetup{labelsep=colon,margin=1.3cm}
\caption{This is a two-dimensional configuration that projects onto the bottom framework of Figure \ref{fig:1}. Dashed lines indicate positive stress (cables) and solid lines indicate negative stress (struts).}
\label{fig:Desargues}
\end{figure}

\begin{figure}[h]
\centering
\includegraphics[scale=0.35]{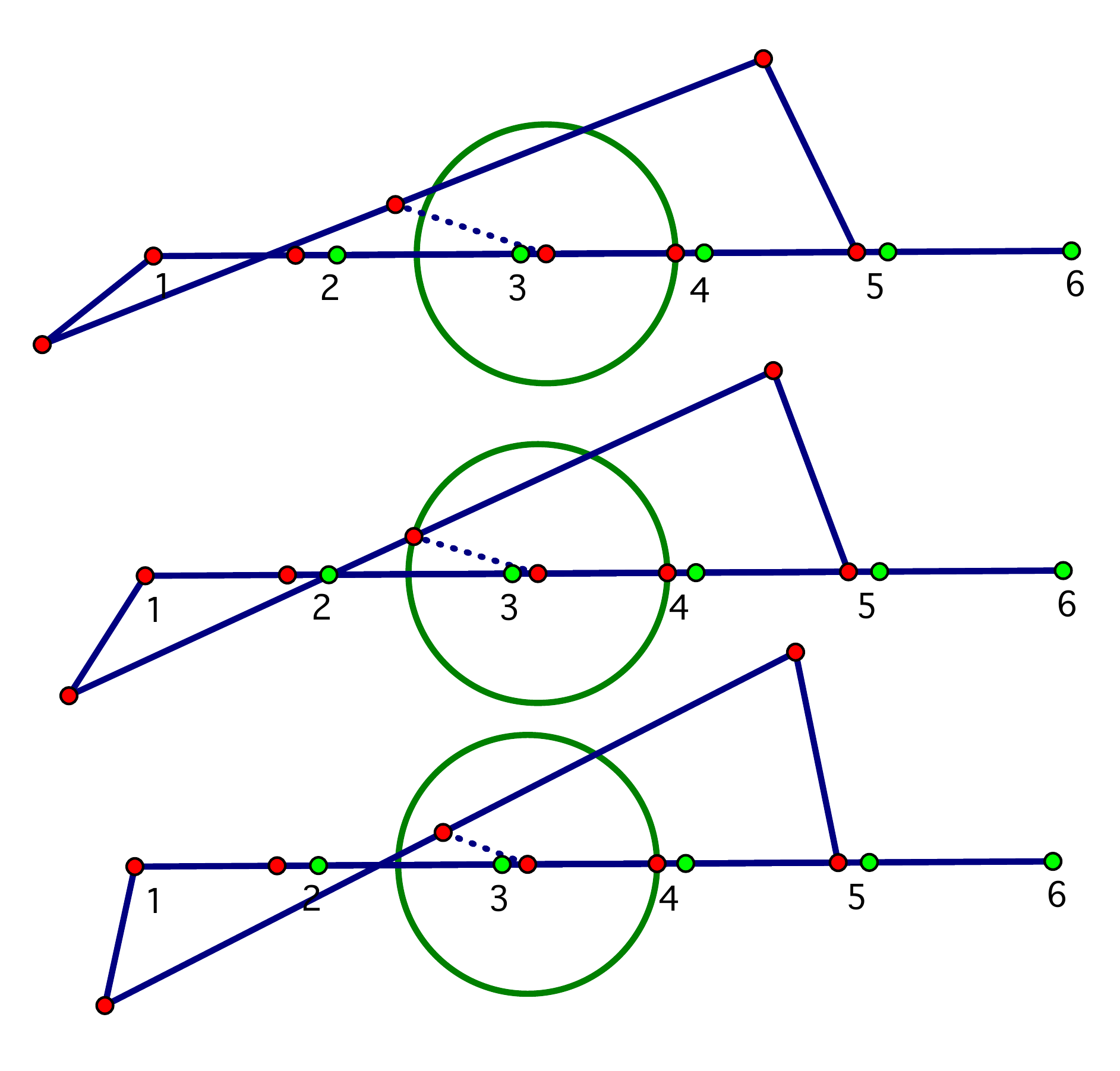}
\captionsetup{labelsep=colon,margin=1.3cm}
\caption{This shows another configuration of the graph in Figure \ref{fig:1} initially on the line, but this configuration has another realization in plane.  Removing the middle edge between the vertices near positions $3$ and $4$ regain their original lengths during the flex as you can see in the middle figure.}
\label{fig:3}
\end{figure}

\paragraph{Universally flexible example}
It is also true at some other configurations of the same graph as in Figures \ref{fig:Desargues} and \ref{fig:1}, that they are not even globally rigid in the plane, let alone universally rigid.  Figure \ref{fig:3} shows that. When the bar joining the vertices close to the points $2$ and $3$ on the line is removed, then the quadrilateral remaining flexes in the plane. The green points on the line are exactly at the integer points $1,2,3,4,5,6$, while the red points on the line are displaced slightly as shown.  Points $1$ and $6$ are not displaced.  The red points on the line and in the three given configurations are shown and accurate.
The green points serve as a ruler, and it is clear in both Figures \ref{fig:Desargues} and \ref{fig:3} that they can be realized at generic configurations in the line.
One can see from the top of Figure \ref{fig:3} that the vertex near $4$ lies outside the circle centered at the point near $3$ with radius equal to the distance between the vertices near $3$ and $4$.  In the bottom of  Figure  \ref{fig:3} it is inside that circle, so at some point it must be on the circle, and that is the alternate realization of the whole ladder framework.

\bigskip

\noindent \textbf{Further Remark:}

If one is happy just to show that frameworks, such as the lower framework of Figure \ref{fig:3}, are universally rigid, it follows quickly from an example, Figure 8, of the ``orchard ladder" in \cite{iterative}. There it is shown that the orchard ladder is universally rigid in the plane.  By Corollary \ref{cor:global}, it is universally rigid in the line.  Figure \ref{fig:orchard} shows such an example.

\begin{figure}[h]
\centering
\includegraphics[scale=0.35]{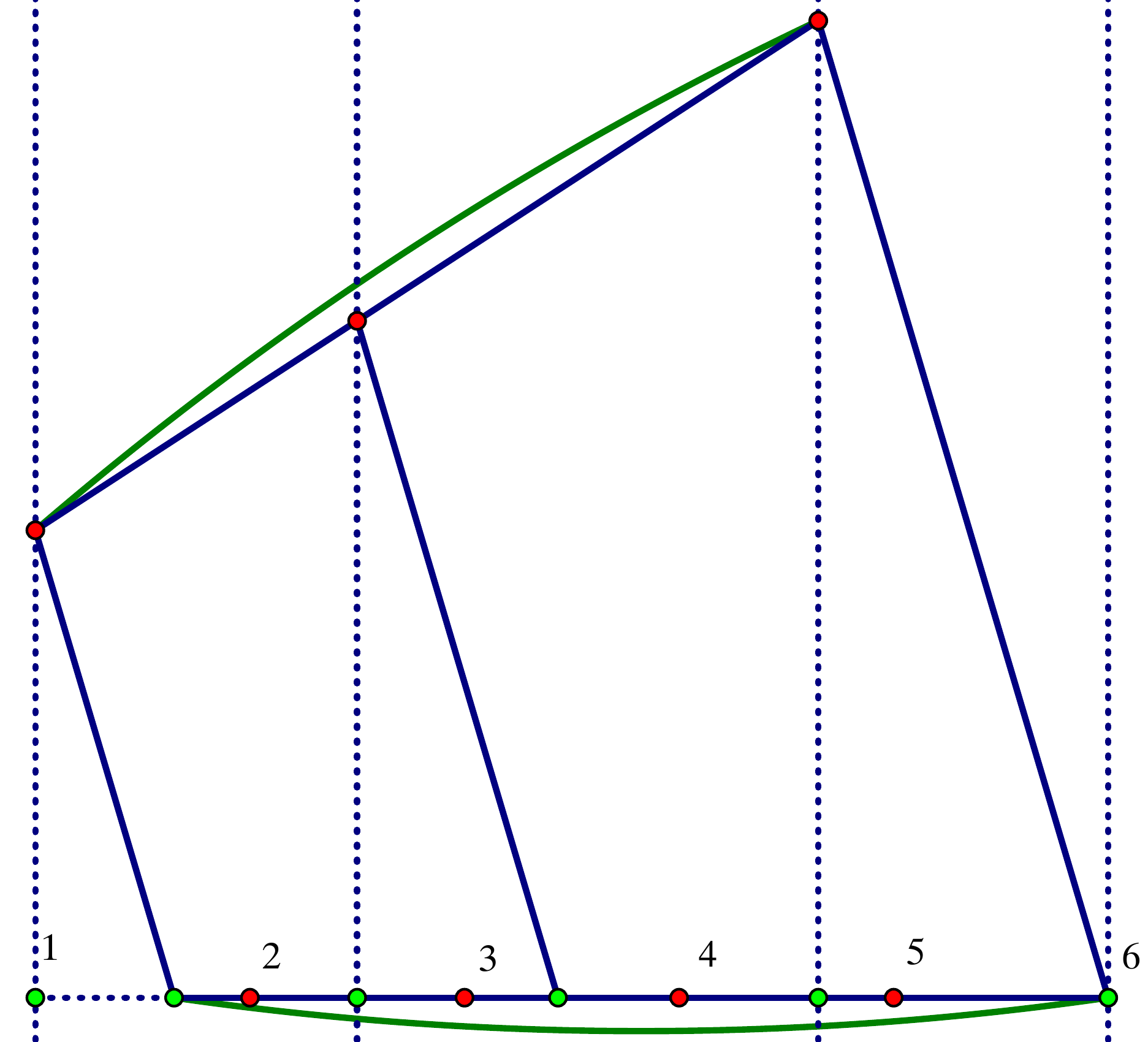}
\captionsetup{labelsep=colon,margin=1.3cm}
\caption{ In this example of an orchard ladder, the rungs are parallel, but just to be universally rigid, it is only necessary that the lines (corresponding to $1,2:3,4:5,6)$ through the three rungs meet at a point. The green points are projections into the line.}
\label{fig:orchard}
\end{figure}

\section{Final questions}


One may also ask: does a universally rigid framework in the line only come from some overlapping stretched cycles, even if some other stresses are needed?

A related question:  If a framework in the line is universally rigid, is it necessarily super stable?  We have no counterexamples.\\

\textbf{Acknowledgement.} We are grateful to the International Centre for Mathematical Sciences who provided partial financial support through a Research in Groups grant.

\bibliographystyle{abbrvnat}
\bibliography{icms.bib}
\end{document}